\documentclass[a4paper,12pt,leqno]{amsart}
\makeatletter
\usepackage{amsfonts,delarray,amssymb,amsmath,amsthm,a4,a4wide}

\usepackage{color}

\title{On the extension of the mean curvature flow}

\author{Nam Q.  Le$^*$}

\address{Department of
Mathematics, Columbia University, New York,
 USA}
\email{namle@math.columbia.edu}

\author{Natasa Sesum$^{**}$}
\address{Department of Mathematics, Columbia University, New York,
USA}
\email{natasas@math.columbia.edu}

\thanks{$**:$ Partially supported
by NSF grant 0604657}

\newcommand{\review}[2][\right]{\relax
\ifx#1\right\relax \left.\fi#2#1\rvert}
\let\abs=\envert
 
\newtheorem{theorem}{Theorem}[section]
\newtheorem{propo}{Proposition}[section]
\newtheorem{remark}{Remark}[section]
\newtheorem{corollary}{Corollary}[section]
\newtheorem{lemma}{Lemma}[section]
\newtheorem{example}{Example}[section]
\newcommand{\bef}{\begin{flushright}}
\newcommand{\eef}{\end{flushright}}

\newcommand{\eval}[2][\right]{\relax
\ifx#1\right\relax \left.\fi#2#1\rvert}

\let\abs=\envert

\numberwithin{equation}{section}
\let\norm=\enVert
\newcommand\e{\varepsilon}
\newcommand{\h}{\hspace*{.24in}}

\def\h{\hspace*{.24in}}

\def\beq{\begin{eqnarray*}}
\def\eeq{\end{eqnarray*}}

\def\RR{\mbox{$I\hspace{-.06in}R$}}

\begin{document}
\maketitle
\author
\pagenumbering{arabic}

\begin{abstract}
Consider a family of smooth immersions $F(\cdot,t): M^n\to \mathbb{R}^{n+1}$ of closed hypersurfaces in $\mathbb{R}^{n+1}$ moving 
by the mean curvature flow $\frac{\partial F(p,t)}{\partial t} = -H(p,t)\cdot \nu(p,t)$, for $t\in [0,T)$.  In \cite{Cooper} Cooper has 
recently  proved that the mean curvature blows up at the singular time $T$. We show that if the second fundamental form stays bounded 
from below all the way to $T$, then the scaling invariant mean curvature integral bound is enough to extend the flow past time $T$, and 
this integral bound is optimal in some sense explained below.
\end{abstract}
\noindent

\section{Introduction}
\h Let $M^{n}$ be a compact $n$-dimensional hypersurface without boundary, and let $F_{0}: M^{n}\rightarrow \RR^{n+ 1}$ be a smooth immersion of $M^{n}$ into $\mathbb{R}^{n+1}$. Consider a smooth one-parameter family of immersions
\begin{equation*}
 F(\cdot, t): M^{n}\rightarrow \RR^{n +1}
\end{equation*}
satisfying
\begin{equation*}
 F(\cdot, 0) = F_{0}(\cdot)
\end{equation*}
and 
\begin{equation}
 \frac{\partial F(p, t)}{\partial t} = -H(p, t)\nu(p,t)~\forall (p, t)\in M\times [0, T).
\label{MCF1}
\end{equation}
Here $H(p, t)$ and $\nu(p, t)$ denote the mean curvature and a choice of unit normal for the hypersurface $M_{t} = F(M^{n},t)$ at $F(p, t)$.
We will sometimes also write $x(p, t) = F(p, t)$ and refer to (\ref{MCF1}) as to the mean curvature flow equation. \\
\h Without any special assumptions on $M_0$, the mean curvature flow (\ref{MCF1}) will in general develop singularities  in finite time, characterized by a blow up of the second fundamental form $A(\cdot,t)$.

\begin{theorem}[Huisken \cite{Huisken84}]
Suppose $T < \infty$ is the first singularity time  for a compact mean curvature flow. Then $\sup_{M_t}|A|(\cdot,t) \to \infty$ as $t\to T$.
\label{Aunbound}
\end{theorem}

By the work of Huisken and Sinestrari \cite{HS} the blow up of $H$ near a singularity is known for mean convex hypersurfaces.  They also established lower bounds on the principal curvatures  in this mean-convex setting. In \cite{Cooper}, by a blowup argument, Cooper shows that the mean curvature being uniformly bounded up to $T < \infty$ is enough to extend the flow (\ref{MCF1}) past time $T$. All those results  motivate a natural question: what are the optimal conditions that will guarantee the existence of a smooth solution to the mean curvature flow (\ref{MCF1})?\\
 
We will use the following notation throughout the whole paper,
$$||v||_{L^p(M\times[0,T))} := (\int_0^T\int_{M_t}|v|^p\, d\mu\, dt)^{\frac{1}{p}},$$
for a function $v(\cdot,t)$ defined on $M\times [0,T)$.

\h In this paper, we prove the following
\begin{theorem}
Assume that for the mean curvature flow (\ref{MCF1}), we have\\
(i) A lower bound on the second fundamental form
\begin{equation}
 h_{ij}\geq -B g_{ij}
\label{pinchA}
\end{equation}
where $B$ is a nonnegative number.\\
(ii) An integral bound on the mean curvature
\begin{equation}
 \norm{H}_{L^{\alpha}(M\times [0, T))}< \infty
\label{intboundH}
\end{equation}
for some $\alpha\geq n + 2$.\\
\h Then the flow can be extended past time $T$.
\label{MCbound}
\end{theorem}

In section \ref{sec-prel} we will show the integral bound assumption (\ref{intboundH}) is optimal in certain sense.

In \cite{WangR} Wang, extending a result of the
second author \cite{sesum}, proved the analogous result for the Ricci flow, namely that if the Ricci curvature is bounded from below, a uniform integral scalar curvature bound is enough to extend the Ricci flow past some finite time.

As can be seen in the proof of Theorem \ref{MCbound} in Section
\ref{thmproof}, the actual conditions we need in lieu of
(\ref{pinchA}) are the
following \\
(iii) A lower bound for the mean curvature
\begin{equation*}
 H\geq -l~ \text{for some} ~l>0
\end{equation*}
and \\
(iv) An upper bound for the squared second fundamental form in terms of a linear function of the squared mean curvature
\begin{equation*}
 \abs{A}^2 \leq C_{\ast}H^2 + b~\text{for some}~ C_{\ast}, b>0.
\end{equation*}
These conditions can be verified in many situations, e.g, for mean convex intitial hypersurfaces $M^{n}$ (see Huisken 
and Sinestrari \cite{HS}) or more generally, all starshaped hypersurfaces and manifolds that can be obtained by buiding in small, concave
dents into mean convex hypersurfaces (see Smoczyk \cite{Smoczyk}).

As a corollary we obtain the following result.

\begin{corollary}
\label{cor-appl}
Let $M^n$ be a mean convex or a starshaped hypersurface in $\mathbb{R}^{n+1}$. Assume we have
$$||H||_{L^{\alpha}(M\times[0,T))} < \infty,$$
for some $\alpha \ge n+2$, along the flow (\ref{MCF1}). The flow can be extended past time $T$.
\end{corollary}
The proof of Theorem \ref{MCbound} is based on a blow-up argument, 
and the Moser iteration using the  Michael-Simon inequality. By their inequality
there is a uniform constant $c_{n}$, depending only on $n$, such that for any nonnegative, $C^1$ function 
$f$ on a hypersurface $M\subset \mathbb{R}^{n+1}$, the
following  holds
\begin{equation}
\label{eqn-simon}
(\int_M f^{\frac{n}{n-1}}\,d\mu)^{\frac{n-1}{n}} \le c_{n}\int_{M}(|\nabla f|+ |H|f)\, d\mu.
\end{equation}
\h The paper is organized as follows. In Section 2, we introduce basic notations concerning evolving hypersurfaces and provide
an example showing that the integral bound (\ref{intboundH}) in Theorem 1.2 is optimal to some extent. In Section 3, we establish a 
modified Michael-Simon inequality and 
Sobolev type inequalities for the mean curvature flow that can be of independent interest. Section 4 is devoted to a reverse 
Holder inequality for a subsolution to a parabolic equation changing during
mean curvature flow. It turns out to be the key estimate for the Moser iteration process carried out in Section 5 
(for the supercritical case) and Section 6 
(for the critical case with a smallness condition). Then we bound uniformly the mean curvature in terms of its integral bounds in Section 7. 
In the final Section \ref{thmproof}, 
we give the proof of the Main Theorem using a blow up argument. 

\section{Preliminaries}
\label{sec-prel}

\h For any compact $n$-dimensional hypersurface $M^{n}$ which is smoothly embedded in $\RR^{n+1}$ by $F: M^{n}\rightarrow \RR^{n+1}$, 
let us denote by $g = (g_{ij})$ the induced metric, $A = (h_{ij})$ the second fundamental form,  $d\mu =\sqrt{\text{det}~(g_{ij})}~dx$ the volume form, 
$\nabla$ the induced Levi-Civita connection
and $\Delta$ the induced Laplacian.
Then the mean curvature of $M^{n}$ is given by
\begin{equation*}
 H = g^{ij}h_{ij}.
\end{equation*}
In \cite{Huisken84} it has been computed that
$$\frac{\partial}{\partial t}d\mu = -H^2\, d\mu,$$
$$\frac{\partial}{\partial t} H = \Delta H + |A|^2 H.$$

To some extent, the constant $\alpha = n +2$ appearing in Theorem \ref{MCbound} is optimal as illustrated by the following example.

\begin{example}
Let $M$ be the standard sphere $S^{n}$ which is immersed into
$\RR^{n +1}$ by $F_{0}$. Then the mean curvature flow with initial
data $M$ has a simple formula: $F(\cdot, t) = r(t)F_{0}(\cdot)$
where $r(t) = \sqrt{1-2nt}$. Therefore $T= \frac{1}{2n}$ is the
extinction time of this mean curvature flow. We have
\begin{equation*}
 r(t) = \sqrt{2n(T-t)}, ~H(t) = \frac{n}{r(t)}.
\end{equation*}
Let us denote by $w_{n}$ the area of $S^{n}$. Compute,
\begin{eqnarray*}
 \norm{H}_{L^{\alpha}(M\times [0, T))} &=& \left(\int_{0}^{T}\frac{n^{\alpha}}{(r(t))^{\alpha}}[r(t)]^nw_{n} dt\right)^{\frac{1}{\alpha}}
\\ &=& \frac{n w_{n}^{\frac{1}{\alpha}}}{(2n)^{\frac{\alpha - n}{2\alpha}}}\left(\int_{0}^{T} \frac{dt}{(T-t)^{\frac{\alpha-n}{2}}}\right)^{\frac{1}{\alpha}}
= \left\{ \begin{aligned} 
<\infty &\h \text{if}~ \alpha < n +2\\
=\infty &\h \text{if} ~ \alpha \geq n + 2.                                                                         
\end{aligned}
\right.
\end{eqnarray*}
Thus, the constant $\alpha$ in (\ref{intboundH}) cannot be smaller than $n +2$.
\end{example}

\begin{remark}
 When $\alpha = n +2$, the quantity $\norm{H}_{L^{\alpha}(M\times [0, T))}$ is invariant under the folowing rescaling of the mean curvature 
flow (\ref{MCF1}):
\begin{equation*}
 \tilde{F}(p, t) = Q F(p, \frac{t}{Q^2}), \,\,\,\text{for} \,\,\,  Q>0.
\end{equation*}
\end{remark}

\begin{remark}
The characterization of the maximal time of existence of a solution to
an evolution equation by the blow up of its scaling invariant
quantities seems to be a  ubiquitous phenomenon. Let us mention a couple of
examples among many. In fluid dynamics, we have the celebrated
Beale-Kato-Majda Theorem \cite{BKM} which says that if the maximal time of existence
of solutions to the  incompressible Euler or Navier-Stokes equation is finite then
necessarily the $L^1_{\mbox{time}}L^{\infty}_{\mathbb{R}^3}$ norm of the
vorticity blows up. In the cases of the Ricci and the  mean
curvature flow,  in addition to controlling the scaling
invariant quantities, we also need lower bounds on the Ricci curvature
and the second fundamental form, respectively. For the incompressible Euler or Navier-Stokes
equations, the divergence-free property of the velocity vector field
plays a crucial role and is in some sense analogous to those lower bounds on the Ricci
curvature and the second fundemantal form.
\end{remark}

\section{Sobolev Inequalities for the Mean Curvature Flow}
\h In this section, we establish a version of Michael-Simon inequality, Lemma \ref{MSlem}, that allows us to derive a 
Sobolev type inequality, Proposition \ref{MStime}, for the mean curvature flow. This Sobolev inequality will be crucial for the Moser iteration in the next sections. The key 
step in the Moser iteration is the inequality (\ref{superc}).\\
\h The following lemma consists of a slightly modified Michael-Simon inequality whose proof is based on the original Michael-Simon inequality
(\ref{eqn-simon}) together with the interpolation inequalities.
\begin{lemma}
Let $M$ be a compact $n$-dimensional hypersurface without boundary, which is smoothly embedded in $\RR^{n + 1}$. Let 
\begin{equation} 
Q = \left\{ \begin{aligned} 
\frac{n}{n-2} &\h \text{if}~ n>2\\
<\infty &\h \text{if} ~ n=2                                                                         
\end{aligned}
\right.
\end{equation}
Then, for all Lipschitz functions $v$ on $M$, we have
\begin{equation*}
 \norm{v}^2_{L^{2Q}(M)}\leq c_{n}\left(\norm{\nabla v}^2_{L^{2}(M)} + \norm{H}^{n + 2}_{L^{n + 2}(M)}\norm{v}^2_{L^{2}(M)}\right)
\end{equation*}
where $H$ is the mean curvature of $M$ and $c_{n}$ is a positive constant depending only on $n$. 
\label{MSlem}
\end{lemma}
\begin{proof}
We only need to prove the lemma for $v\geq 0$. Applying Michael-Simon's inequality (\ref{eqn-simon})\cite{MM} to the function $w = v^{\frac{2(n-1)}{n-2}}$, 
we get
\begin{equation*}
 \left(\int_{M}v^{\frac{2n}{n-2}} d\mu\right)^{\frac{n-1}{n}}\leq c_{n}\left(\int_{M}\abs{\nabla v}v^{\frac{n}{n-2}} d\mu + 
\int_{M}\abs{H}v^{\frac{2(n-1)}{n-2}} d\mu\right).
\end{equation*}
By Holder's inequality it follows that
\begin{eqnarray*}
 \left(\int_{M}v^{\frac{2n}{n-2}}d\mu\right)^{\frac{n-2}{n}}&\leq& c_{n}^{\frac{n-2}{n-1}}\left(\int_{M}\abs{\nabla v}v^{\frac{n}{n-2}} d\mu + 
\int_{M}\abs{H}v^{\frac{2(n-1)}{n-2}} d\mu\right)^{\frac{n-2}{n-1}}\\ &\leq &
c_{n}\left (\norm{\nabla v}_{L^{2}(M)}\norm{v}^{\frac{n}{n-2}}_{L^{2Q}(M)} + 
\norm{H}_{L^{n + 2}(M)}\norm{v}^{\frac{2(n-1)}{n-2}}_{L^{2m}(M)} \right)^{\frac{n-2}{n-1}}\\ &\leq &
c_{n} \left (\norm{\nabla v}^{\frac{n-2}{n-1}}_{L^{2}(M)}\norm{v}^{\frac{n}{n-1}}_{L^{2Q}(M)} + 
\norm{H}^{\frac{n-2}{n-1}}_{L^{n + 2}(M)}\norm{v}^{2}_{L^{2m}(M)} \right).
\end{eqnarray*} 
where
\begin{equation*}
 m =\frac{(n-1)(n + 2)}{(n-2)(n + 1)}.
\end{equation*}
Thus
\begin{equation}
  \norm{v}^2_{L^{2Q}(M)}\leq c_{n} \left (\norm{\nabla v}^{\frac{n-2}{n-1}}_{L^{2}(M)}\norm{v}^{\frac{n}{n-1}}_{L^{2Q}(M)} + 
\norm{H}^{\frac{n-2}{n-1}}_{L^{n + 2}(M)}\norm{v}^{2}_{L^{2m}(M)} \right).
\label{firstMM}
\end{equation}
By Young's inequality
\begin{equation}
 ab = (\e^{1/p}a)(\e^{-1/p}b) \leq \frac{\e a^{p}}{p} +\frac{\e^{-q/p}b^{q}}{q}\leq \e a^{p} + \e^{-q/p}b^{q},
\label{Young}
\end{equation}
where $a,b,\e > 0$, $p, q>1$ and $\frac{1}{p} +\frac{1}{q} =1$. If we apply it to (\ref{firstMM}), with 
\begin{equation*}
 a = \norm{v}^{\frac{n}{n-1}}_{L^{2Q}(M)}, \qquad b = \norm{\nabla v}^{\frac{n-2}{n-1}}_{L^{2}(M)},
\end{equation*}
and \begin{equation*}
    \e = \frac{1}{2c_n}, \,\,\, p =\frac{2(n-1)}{n}, \,\,\, q =\frac{2(n-1)}{n-2},
    \end{equation*}
we obtain
\begin{equation*}
  \norm{v}^2_{L^{2Q}(M)}\leq c_{n} \left ( \frac{1}{2c_{n}}\norm{v}^2_{L^{2Q}(M)} + 
(\frac{1}{2c_{n}})^{\frac{-n}{n-2}}\norm{\nabla v}^{2}_{L^{2}(M)} + 
\norm{H}^{\frac{n-2}{n-1}}_{L^{n + 2}(M)}\norm{v}^{2}_{L^{2m}(M)} \right).
\end{equation*}
Hence
\begin{equation}
  \norm{v}^2_{L^{2Q}(M)}\leq c_{n} \left (\norm{\nabla v}^{2}_{L^{2}(M)} + 
\norm{H}^{\frac{n-2}{n-1}}_{L^{n + 2}(M)}\norm{v}^{2}_{L^{2m}(M)} \right).
\label{beforeinter}
\end{equation}
Next, we will use the following interpolation inequality (see inequality (7.10) in \cite{GT})
\begin{equation}
 \norm{u}_{L^{r}}\leq \e  \norm{u}_{L^{s}} + \e^{-\mu}  \norm{u}_{L^{t}}
\label{interpol}
\end{equation}
where $t<r<s$ and 
\begin{equation*}
 \mu = (\frac{1}{t}-\frac{1}{r})/(\frac{1}{r}-\frac{1}{s}). 
\end{equation*}
Note that, in our case
\begin{equation*}
 1<m<Q,
\end{equation*}
and therefore, by (\ref{interpol})
\begin{equation}
 \norm{v}_{L^{2m}(M)}\leq \e  \norm{v}_{L^{2Q}(M)} + \e^{-\alpha}  \norm{v}_{L^{2}(M)}
\label{interep}
\end{equation}
where $\e>0$ and 
\begin{equation}
\alpha = \frac{Q(m-1)}{Q-m} =  \frac{n^2}{n-2}.
\end{equation}
Plugging (\ref{interep}) into the right hand side of (\ref{beforeinter}), we deduce that
\begin{eqnarray}
 \norm{v}^2_{L^{2Q}(M)}&\leq& c_{n}\norm{\nabla v}^{2}_{L^{2}(M)} + c_{n} \norm{H}^{\frac{n-2}{n-1}}_{L^{n + 2}(M)}
\left ( \e  \norm{v}_{L^{2Q}(M)} + \e^{-\alpha}  \norm{v}_{L^{2}(M)} \right)^2\nonumber \\ &\leq &
c_{n}\norm{\nabla v}^{2}_{L^{2}(M)} + c_{n} \norm{H}^{\frac{n-2}{n-1}}_{L^{n + 2}(M)}
\left ( \e^2  \norm{v}^2_{L^{2Q}(M)} + \e^{-2\alpha}  \norm{v}^2_{L^{2}(M)} \right).
\label{absorb1}
\end{eqnarray}
Now, we can absorb the term involving $ \norm{v}^2_{L^{2Q}(M)}$ into the left hand side of (\ref{absorb1}) by choosing
\begin{equation*}
 \e^{2} =\frac{1}{2c_{n}}\norm{H}^{-\frac{n-2}{n-1}}_{L^{n + 2}(M)}.
\end{equation*}
Since $\frac{n-2}{n-1}(1 + \alpha) =  n + 2$, we obtain the desired inequality
\begin{equation*}
  \norm{v}^2_{L^{2Q}(M)}\leq c_{n}\norm{\nabla v}^{2}_{L^{2}(M)} +
 c_{n}\norm{H}^{n+2}_{L^{n + 2}(M)}\norm{v}^2_{L^{2}(M)}. 
\end{equation*}

\end{proof}
\h Our Sobolev type inequality for the mean curvature flow is stated in the following proposition.

\begin{propo}
For all nonnegative Lipschitz functions $v$, one has
\begin{multline}
 ||v||^{\beta}_{L^{\beta}(M\times[0,T))}\\ \leq c_{n}\max_{0\leq t\leq T}
\norm{v}^{4/n}_{L^{2}(M_{t})}\left (||\nabla v||^2_{L^2(M\times[0,T))} +  \max_{0\leq t\leq T}
\norm{v}^{2}_{L^{2}(M_{t})} ||H||^{n+2}_{L^{n+2}(M\times[0,T))}\right),
\end{multline}
where $\beta := \frac{2(n+2)}{n}$.
\label{MStime}
 \end{propo}
\begin{proof}
 By Holder's inequality, we have
\begin{eqnarray*}
 \int_{0}^{T}\int_{M_{t}}v^{\frac{2(n+2)}{n}}d\mu dt &=& \int_{0}^{T}dt \int_{M_{t}} v^2 v^{4/n}d\mu\\ &\leq&
\int_{0}^{T}dt \left(\int_{M_{t}}v^{\frac{2n}{n-2}}d\mu\right)^{\frac{n-2}{n}}\left(\int_{M_{t}}v^{2}d\mu\right)^{\frac{2}{n}}\\ &\leq &
\max_{0\leq t\leq T}
\norm{v}^{4/n}_{L^{2}(M_{t})} \int_{0}^{T}\norm{v(\cdot, t)}^{2}_{L^{2Q}(M_{t})}.
\end{eqnarray*}
Now, applying Lemma \ref{MSlem}, we get
\begin{multline*}
||v||^{\beta}_{L^{\beta}(M\times[0,T))}\\ \leq c_{n}\max_{0\leq t\leq T}
\norm{v}^{4/n}_{L^{2}(M_{t})}\left (\int_0^T\int_{M_t}|\nabla v|^2\, d\mu\, dt + 
\int_0^T(\int_{M_t}|H|^{n+2}\, d\mu)||v(\cdot,t)||^2_{L^2(M_t)}\,  dt\right)\\
\leq c_{n}\max_{0\leq t\leq T}
\norm{v}^{4/n}_{L^{2}(M_{t})}\left (||\nabla v||^2_{L^2(M\times[0,T))} +  \max_{0\leq t\leq T}
\norm{v}^{2}_{L^{2}(M_{t})} ||H||^{n+2}_{L^{n+2}(M\times[0,T))}\right).
\end{multline*}
\end{proof}

\section{A Reverse Holder Inequality}
\h In this section, we establish a soft version of reverse Holder
inequality for parabolic inequality during the mean curvature flow.
Because of the blow up argument that we will use in the end, we should
keep track of our constants in deriving the estimates, since
we do not want them to blow up after taking the limit of the rescaled flow. Moreover, since we also need certain smallness conditions
to carry out the Moser iteration, we do not want constants quantifying these smallness conditions to vanish after taking the limit of 
the rescaled flow. Here is the convention that we will use.\\
\h {\it Constants such as $C_0, C_1, C_2, \dots$ will be defined. The constants
with alphabetical subscripts $C_{a}, C_{b},\dots$ depend on other
constants with numerical subscripts $C_{0}, C_{1},\dots$ in a
controlled way, the former are increasing functions of the later.  The
$\delta$-constants with numerical subscripts, such as $\delta_{1},
\delta_{2}, \dots$ depend on the constants with numerical subscripts
$C_{0}, C_{1}, \dots$, the former are decreasing functions of the
latter.  We will use those facts in the
blow up argument in section \ref{thmproof}.}\\

We start with the differential inequality
\begin{equation}
 (\frac{\partial}{\partial t}-\Delta ) v\leq fv, ~v\geq 0
\label{keyeq}
\end{equation}
where the function $f$ has bounded $L^{q}(M\times [0, T))$-norm with $q \ge \frac{n+2}{2}$.
Let $\eta(t,x)$ be a smooth function with the property that $\eta(0, x) = 0$ for all $x$. 
\begin{lemma}
Let 
\begin{equation}
 C_{0} \equiv C_{0}(q)= ||f||_{L^q(M\times[0,T))},\qquad  C_{1} = (1 +\norm{H}^{n+2}_{L^{n + 2}(M\times [0, T))})^{\frac{n}{n + 2}},
\label{Czero}
\end{equation}
$\beta > 1$ be a fixed number and $q>\frac{n +2}{2}$. 
Then there exists a positive constant $C_{a} = C_{a}(n, q, C_{0}, C_{1})$ such that
 \begin{equation}
 ||\eta^2 v^{\beta}||_{L^{(n+2)/n}(M\times[0,T))} 
\leq C_{a}\Lambda(\beta)^{1 + \nu}||v^{\beta}\left(\eta^2 + \abs{\nabla\eta}^2 + 
2\eta \abs{(\frac{\partial}{\partial t}-\Delta)\eta}\right)||_{L^1(M\times[0,T))},
\label{RH}
\end{equation}
where 
\begin{equation}
 \nu =\frac{n + 2}{ 2q -(n +2)},
\end{equation}
and $\Lambda(\beta)$ is a positive constant depending on $\beta$ such that $\Lambda(\beta)\geq 1$ if 
$\beta\geq 2$ (e.g. we can choose $\Lambda(\beta) = 100\beta$).
\label{softRH}
\end{lemma}
\begin{remark}
 As will be seen later, we can choose
\begin{equation}
 C_{a}(n, q, C_{0}, C_{1}) = (2c_{n}C_{0}C_{1})^{1+\nu}.
\end{equation}
\end{remark}

\begin{proof}
 We use $\eta^2 v^{\beta-1}$ as a test function in the inequality
\begin{equation*}
 -\Delta v +\frac{\partial v}{\partial t}\leq fv.
\end{equation*}
It follows that, for any $s\in (0, T]$, we have
\begin{equation}
 \int_{0}^{s}\int_{M_{t}}(-\Delta v)\eta^2 v^{\beta-1} d\mu dt + 
\int_{0}^{s}\int_{M_{t}} \frac{\partial v}{\partial t}\eta^2 v^{\beta-1}d\mu dt
\leq  \int_{0}^{s}\int_{M_{t}} \abs{f}\eta^2 v^{\beta}d\mu dt.
\label{int1}
\end{equation}
Note that, by integrating by parts
\begin{equation}
 \int_{M_{t}} (-\Delta v) \eta^2 v^{\beta-1}d\mu  = \int_{M_{t}}2<\nabla v, \nabla\eta>\eta v^{\beta-1} d\mu 
+ (\beta-1)\int_{M_{t}}\eta^2 v^{\beta-2}\abs{\nabla v}^2 d\mu.
\label{int2}
\end{equation}
Using the evolution of the volume form 
\begin{equation*}
 \partial_{t} d\mu = -H^2 d\mu
\end{equation*}
and recalling the properties of $\eta$, we get
\begin{eqnarray}
 \int_{0}^{s}\int_{M_{t}} \frac{\partial v}{\partial t}\eta^2 v^{\beta-1}d\mu dt &=& \frac{1}{\beta}
\int_{0}^{s}\int_{M_{t}} \frac{\partial( v^{\beta})}{\partial t}\eta^2 d\mu dt\nonumber \\&=&
\frac{1}{\beta}\int_{M_{t}}v^{\beta}\eta^{2}d\mu\mid_{0}^{s}-\frac{1}{\beta}\int_{0}^{s}\int_{M_{t}}  v^{\beta}\partial_{t}(\eta^2 d\mu)dt
\nonumber\\
&=& \frac{1}{\beta}\int_{M_{s}}v^{\beta}\eta^{2}d\mu -\frac{1}{\beta}\int_{0}^{s}\int_{M_{t}}  v^{\beta}\left[
2\eta \frac{\partial \eta}{\partial t} - H^2 \right]d\mu dt.
\label{int3}
\end{eqnarray}
Therefore, we deduce from (\ref{int1})-(\ref{int3}) the following inequality
\begin{eqnarray}
\label{ibpart1}
& & \int_{0}^{s}\int_{M_{t}}\left(2<\nabla v, \nabla\eta>\eta v^{\beta-1} + (\beta-1)\eta^2 v^{\beta-2}\abs{\nabla v}^2\right)d\mu dt + 
\frac{1}{\beta}\int_{M_{s}}v^{\beta}\eta^{2}d\mu\\ 
&\leq& \frac{1}{\beta}\int_{0}^{s}\int_{M_{t}}  v^{\beta}
2\eta \frac{\partial \eta}{\partial t}d\mu dt  + 
 \int_{0}^{s}\int_{M_{t}} \abs{f}\eta^2 v^{\beta}d\mu dt. \nonumber
\end{eqnarray}
As will be seen later, because we can get good control of the quantity $(\frac{\partial}{\partial t} -\Delta) \eta $
for suitable choices of $\eta$, it is more convenient to make this 
term appear on the right hand side of (\ref{ibpart1}). Observe that, integrating by parts yields
\begin{eqnarray*}
 \frac{1}{\beta}\int_{0}^{s}\int_{M_{t}}  v^{\beta}
2\eta \frac{\partial \eta}{\partial t}d\mu dt &=& \frac{1}{\beta}\int_{0}^{s}\int_{M_{t}} \left( v^{\beta}
2\eta (\frac{\partial}{\partial t} -\Delta) \eta +  v^{\beta}
2\eta\Delta \eta)\right)d\mu dt\\ &=&
\frac{1}{\beta}\int_{0}^{s}\int_{M_{t}}  \left(v^{\beta}
2\eta (\frac{\partial}{\partial t} -\Delta) \eta  -2\nabla (v^{\beta} \eta)\nabla\eta\right)d\mu dt\\ &=&
\frac{1}{\beta}\int_{0}^{s}\int_{M_{t}}  \left(v^{\beta}
2\eta (\frac{\partial}{\partial t} -\Delta) \eta - 2v^{\beta}\abs{\nabla\eta}^2-
2\beta<\nabla v, \nabla\eta>\eta v^{\beta-1}\right)d\mu dt\\&\leq &\frac{1}{\beta}\int_{0}^{s}\int_{M_{t}}  v^{\beta}
2\eta (\frac{\partial}{\partial t} -\Delta) \eta d\mu dt -\int_{0}^{s}\int_{M_{t}}2\eta<\nabla v,\nabla\eta>v^{\beta-1}d\mu dt.
\end{eqnarray*}
Then (\ref{ibpart1}) implies
\begin{multline}
 \int_{0}^{s}\int_{M_{t}}\left(4<\nabla v, \nabla\eta>\eta v^{\beta-1} + (\beta-1)\eta^2 v^{\beta-2}\abs{\nabla v}^2\right)d\mu dt + 
\frac{1}{\beta}\int_{M_{s}}v^{\beta}\eta^{2}d\mu \\  \leq \frac{1}{\beta}\int_{0}^{s}\int_{M_{t}}  v^{\beta}
2\eta \abs{(\frac{\partial}{\partial t} -\Delta) \eta} d\mu dt + 
 \int_{0}^{s}\int_{M_{t}} \abs{f}\eta^2 v^{\beta}d\mu dt.
\label{ibpart2}
\end{multline}
Using the Cauchy-Schwartz inequality
\begin{equation*}
 \int_{0}^{s}\int_{M_{t}}4<\nabla v, \nabla\eta>\eta v^{\beta-1}d\mu dt \geq 
-2\e^2 \int_{0}^{s}\int_{M_{t}} \eta^{2}v^{\beta-2}\abs{\nabla v}^2 d\mu dt-
\frac{2}{\e^2}\int_{0}^{s}\int_{M_{t}} v^{\beta}\abs{\nabla \eta}^2 d\mu dt
\end{equation*}
we get from (\ref{ibpart2}),
\begin{multline}
 \int_{0}^{s}\int_{M_{t}} (\beta-1-2\e^2)\eta^2 v^{\beta-2}\abs{\nabla v}^2 d\mu dt + 
\frac{1}{\beta}\int_{M_{s}}v^{\beta}\eta^{2}d\mu\\ \leq \frac{1}{\beta}\int_{0}^{s}\int_{M_{t}}  v^{\beta}
2\eta \abs{(\frac{\partial}{\partial t} -\Delta) \eta}d\mu dt  + 
 \int_{0}^{s}\int_{M_{t}} \abs{f}\eta^2 v^{\beta}d\mu dt + \frac{2}{\e^2}\int_{0}^{s}\int_{M_{t}} v^{\beta}\abs{\nabla \eta}^2 d\mu dt.
\label{ibpart3}
\end{multline}
Choosing $\e^2 =\frac{\beta-1}{4}$ and observing that $\abs{\nabla (v^{\beta/2})}^2 =\frac{\beta^2}{4}v^{\beta-2}\abs{\nabla v}^2$
yield
\begin{multline*}
2(1-\frac{1}{\beta})\int_{0}^{s}\int_{M_{t}} \eta^2 \abs{\nabla (v^{\beta/2})}^2 d\mu dt + 
\int_{M_{s}}v^{\beta}\eta^{2}d\mu\\ \leq \int_{0}^{s}\int_{M_{t}}  v^{\beta}
2\eta \abs{(\frac{\partial}{\partial t} -\Delta) \eta}d\mu dt  + 
 \beta\int_{0}^{s}\int_{M_{t}} \abs{f}\eta^2 v^{\beta} d\mu dt
+ \frac{8\beta}{\beta-1}\int_{0}^{s}\int_{M_{t}} v^{\beta}\abs{\nabla \eta}^2d\mu dt.
\end{multline*}
Combining the previous estimate with
\begin{equation*}
 \abs{\nabla(\eta v^{\beta/2})}^2\leq 2\eta^2 \abs{\nabla (v^{\beta/2})}^2 + 2v^{\beta}\abs{\nabla\eta}^2
\end{equation*}
implies
\begin{multline*}
(1-\frac{1}{\beta})\int_{0}^{s}\int_{M_{t}}\abs{\nabla(\eta v^{\beta/2})}^2 d\mu dt + 
\int_{M_{s}}v^{\beta}\eta^{2}d\mu\\ \leq \int_{0}^{s}\int_{M_{t}}  v^{\beta}
2\eta \abs{(\frac{\partial}{\partial t} -\Delta) \eta}d\mu dt  + 
 \beta\int_{0}^{s}\int_{M_{t}} \abs{f}\eta^2 v^{\beta}d\mu dt + 8(\frac{\beta}{\beta-1} +\frac{\beta-1}{\beta})
\int_{0}^{s}\int_{M_{t}} v^{\beta}\abs{\nabla \eta}^2 d\mu dt.
\end{multline*}
It follows that, for some $\Lambda(\beta)\geq 1$ (say $\Lambda (\beta) = 100\beta$ if $\beta\geq 2$),
\begin{multline*}
\int_{0}^{s}\int_{M_{t}}\abs{\nabla(\eta v^{\beta/2})}^2d\mu dt  + 
\int_{M_{s}}v^{\beta}\eta^{2}d\mu\\ \leq \Lambda(\beta)\left(\int_{0}^{s}\int_{M_{t}}  v^{\beta}\left\{
2\eta \abs{(\frac{\partial}{\partial t} -\Delta) \eta}  +  \abs{\nabla \eta}^2\right\}d\mu dt + 
 \int_{0}^{s}\int_{M_{t}} \abs{f}\eta^2 v^{\beta}d\mu dt \right)\\ \leq 
\Lambda(\beta)\left(\int_{0}^{s}\int_{M_{t}}  v^{\beta}\left\{
2\eta \abs{(\frac{\partial}{\partial t} -\Delta) \eta}  + \abs{\nabla \eta}^2\right\}d\mu dt + 
 ||f||_{L^{q}(M\times[0,T))} ||\eta^2 v^{\beta}||_{L^{\frac{q}{q-1}}(M\times[0,T))} \right)\\
= \Lambda(\beta)\left(\int_{0}^{s}\int_{M_{t}}  v^{\beta}\left\{
2\eta \abs{(\frac{\partial}{\partial t} -\Delta) \eta}  +  \abs{\nabla \eta}^2 \right\} d\mu dt+ 
 C_{0}||\eta^2 v^{\beta}||_{L^{\frac{q}{q-1}}(M\times[0,T))} \right) =: A.
\end{multline*}
Consequently,
\begin{equation}
 \max_{0\leq s\leq T}\int_{M_{s}}\eta^2v^{\beta}d\mu\leq A
\end{equation}
and
\begin{equation}
 \int_{0}^{T}\int_{M_{t}}\abs{\nabla(\eta v^{\beta/2})}^2 d\mu dt\leq A.
\end{equation}
We are now in a position to apply Proposition \ref{MStime} to $\eta v^{\beta/2}$ and get the following estimates
\begin{eqnarray*}
 & & ||\eta^2v^{\beta}||^{(n+2)/n}_{L^{(n+2)/n}(M\times[0,T))} = ||\eta v^{\beta/2}||^{2(n+2)/n}_{L^{2(n+2)/n}(M\times[0,T))}\\ 
& &\leq c_{n}\max_{0\leq t\leq T}
\norm{\eta v^{\beta/2}}^{4/n}_{L^{2}(M_{t})}\left (||\nabla(\eta v^{\beta/2})||^2_{L^2(M\times[0,T))} +  \max_{0\leq t\leq T}
\norm{\eta v^{\beta/2}}^{2}_{L^{2}(M_{t})} ||H||^{n+2}_{L^{n+2}(M\times[0,T))}\right)\\
& & \leq c_{n}A^{2/n} \left (A + A ||H||^{n +2}_{L^{n+2}(M\times[0,T))}\right)
= c_{n}A^{\frac{n + 2}{n}} (1 + ||H||^{n+2}_{L^{n+2}(M\times[0,T))}).
\end{eqnarray*}
Let $S := M\times [0, T)$ and let the norm  $||\cdot||_{L^p(M\times[0,T))}$ be shortly denoted by $||\cdot||_{L^p(S)}$. Then the previous estimate,
using a definition of $A$, can be rewritten as
\begin{eqnarray}
 \norm{\eta^2 v^{\beta}}_{L^{\frac{n + 2}{n}}(S)}
 &\leq& c_{n}A(1 +\norm{H}^{n+2}_{L^{n + 2}(S)})^{\frac{n}{n + 2}} \nonumber \\
 &=& c_{n}C_{1}\Lambda (\beta )\left(\int_{0}^{T}\int_{M_{t}}  v^{\beta}\left\{
2\eta \abs{(\frac{\partial}{\partial t} -\Delta) \eta}  + \abs{\nabla \eta}^2\right\}d\mu dt + 
 C_{0}\norm{\eta^2 v^{\beta}}_{L^{\frac{q}{q-1}}(S)} \right).
\label{superc}
\end{eqnarray}
Since $1 < \frac{q}{q-1} < \frac{n+2}{n}$, by using the interpolation inequality
\begin{equation*}
 \norm{\eta^2 v^{\beta}}_{L^{\frac{q}{q-1}}(S)} \leq \e \norm{\eta^2 v^{\beta}}_{L^{\frac{ n +2}{n}}(S)} + \e^{-\nu}
\norm{\eta^2 v^{\beta}}_{L^{1}(S)}
\end{equation*}
in (\ref{superc}), for $ \nu =\frac{n + 2}{ 2q -(n +2)}$,  one gets 
\begin{eqnarray*}
 & &[1- c_{n}\Lambda(\beta)C_{0}C_{1}\e]\norm{\eta^2 v^{\beta}}_{L^{\frac{n +2}{n}}(S)}\\ 
 &\leq& c_{n}C_{1}\Lambda (\beta)\left[
C_{0}\e^{-\nu}\norm{\eta^2 v^{\beta}}_{L^{1}(S)} + \norm{v^{\beta}(\abs{\nabla\eta}^2 + 
2\eta (\frac{\partial}{\partial t}-\Delta)\eta)}_{L^{1}(S)}\right].
\end{eqnarray*}
If we choose $ \e=\frac{1}{2\Lambda (\beta)c_{n}C_{0}C_{1}}$, then
\begin{eqnarray*}
 & &\norm{\eta^2 v^{\beta}}_{L^{\frac{n +2}{n}}(S)} \\
&\leq& 2c_{n}C_{1}\Lambda (\beta)\left[C_{0}(2\Lambda (\beta)c_{n}C_{0}C_{1})^{\nu}\norm{\eta^2 v^{\beta}}_{L^{1}(S)} 
+ \norm{v^{\beta}(\abs{\nabla\eta}^2 + 
2\eta (\frac{\partial}{\partial t}-\Delta)\eta)}_{L^{1}(S)}\right]\\
&\leq& C_{a}(n, q, C_{0}, C_{1})\Lambda(\beta)^{1 + \nu} \norm{v^{\beta}(\eta^2 + \abs{\nabla\eta}^2 +  2\eta (\frac{\partial}{\partial t}-\Delta)\eta)}_{L^{1}(S)},
\end{eqnarray*}
where  $C_{a}(n, q, C_{0}, C_{1}) = (2c_{n}C_{0}C_{1})^{1+\nu}$.
In conclusion, we get a soft reverse Holder inequality
\begin{equation*}
||\eta^2 v^{\beta}||_{L^{\frac{n+2}{n}}(S)}\le C_a(n,q,C_0,C_1)\Lambda(\beta)^{1+\nu}||v^{\beta}\left(\eta^2 + \abs{\nabla\eta}^2 + 
2\eta \abs{(\frac{\partial}{\partial t}-\Delta)\eta}\right)||_{L^1(S)}.
\end{equation*}
\end{proof}

\section{The Moser Iteration Process for the Supercritical Case}

We will use the notation from previous sections.
Consider the function $v$, which is a solution to (\ref{keyeq}), where $f\in  L^q(S)$. Assume $q > \frac{n+2}{n}$ which corresponds to a supercritical case. 
We will show in this case that an $L^{\infty}$-norm of $v$ over a smaller set can be bounded by an $L^{\beta}$-norm of $v$ on a bigger set, 
where $\beta \ge 2$. Fix $x_{0}\in \RR^{n +1}$. Consider the following sets in space and time,
\begin{equation*}
 D = \cup_{0\le t\le 1} (B(x_{0}, 1)\cap M_t);\,\,\, D^{'} = \cup_{\frac{1}{12}\le t\leq 1}(B(x_{0}, \frac{1}{2})\cap M_t).
\end{equation*}
Let us denote by
\begin{equation*}
 D_{k} = \cup_{t_k\le t\le 1}(B(x_{0}, r_{k})\cap M_t)
\end{equation*}
where
\begin{equation*}
 r_{k} = \frac{1}{2} + \frac{1}{2^{k +1}}; \,\,\,t_{k} = \frac{1}{12}(1-\frac{1}{4^{k}}).
\end{equation*}
Then,
\begin{equation*}
 \rho_{k}:= r_{k-1}-r_{k} = \frac{1}{2^{k +1}}; \,\,\, t_{k}-t_{k-1} = \rho^{2}_{k}.
\end{equation*}
Let us choose a test function $\eta_{k} = \eta_{k}(t,x)$, following Ecker \cite{Ecker95}, of the form
\begin{equation}
 \eta_{k}(t,x) =\varphi_{\rho_{k}}(t)\times \psi_{\rho_{k}}(\abs{x-x_{0}}^2).
\label{testk}
\end{equation}
In (\ref{testk}), the function $\varphi_{\rho_{k}}$ satisfies
\begin{equation*}
 \varphi_{\rho_{k}}(t)= 
\left\{
 \begin{alignedat}{1}
1 \h&\text{if}~ t_{k}\leq t\leq 1, \\\
\in [0, 1] \h &\text{if}~ t_{k-1}\leq t\leq t_{k}, \\\
0\h &\text{if}~ t\leq t_{k-1}.
 \end{alignedat} 
 \right.
 \end{equation*}
and
\begin{equation*}
 \abs{\varphi^{'}_{\rho_{k}}}(t)\leq \frac{c_{n}}{\rho^2_{k}};
\end{equation*}
whereas in (\ref{testk}), the function 
$\psi_{\rho_{k}}(s)$ satisfies
\begin{equation*}
 \psi_{\rho_{k}}(s)= 
\left\{
 \begin{alignedat}{1}
0 \h&\text{if}~  s\geq r_{k-1}^2, \\\
\in [0, 1] \h &\text{if}~ r^2_{k}\leq s\leq r^{2}_{k-1}, \\\
1\h &\text{if}~ s\leq r^2_{k}.
 \end{alignedat} 
 \right.
 \end{equation*}
and
\begin{equation*}
 \abs{\psi^{'}_{\rho_{k}}}(s)\leq \frac{c_{n}}{\rho^2_{k}}.
\end{equation*}
We have
\begin{equation}
 0\leq \eta_{k}\leq 1; \eta_{k}\equiv 1~\text{in}~D_{k}; \eta_{k}\equiv 0~\text{outside}~D_{k-1}.
\label{etaDk}
\end{equation}
\h Using the following identity for the mean curvature flow derived in Brakke \cite{Brakke}
\begin{equation}
 (\frac{d}{dt}-\Delta) \abs{x -x_{0}}^2 = -2n~\forall x\in M_{t},
\end{equation}
we can verify the following
\begin{lemma}
\begin{equation}
 \sup_{t\in [0, 1]}\sup_{x\in M_{t}} \left(\eta^2_{k}(t,x) + \abs{\nabla\eta_{k}(t,x)}^2 + 
2\eta_{k} (t,x)\abs{(\frac{\partial}{\partial t}-\Delta)\eta_{k}(t,x)}\right)\leq \frac{c_{n}}{\rho^2_{k}} = c_{n}4^{k +1}.
\end{equation}
\label{heatest}
\end{lemma}
\h The main result of this section is the following Harnack inequality in the supercritical case.
\begin{lemma}
Consider the equation (\ref{keyeq}) with $T\geq 1$. Let us denote by
$ \lambda = \frac{n +2}{n}$, let $q>\frac{n +2}{2}$ and $\beta\geq 2$. Then, there exists a constant $C_{b}= C_{b}(n, q, \beta, C_{0}, C_{1})$ such that
\begin{equation}
 \norm{v}_{L^{\infty}(D^{'})}\leq C_{b}(n, q, \beta, C_{0}, C_{1}) \norm{v}_{L^{\beta}(D)},
\label{firstM}
\end{equation}
and
\begin{equation}
 \norm{v}_{L^{\infty}(D^{'})}\leq C_{b}(n, q, \beta, C_{0}, C_{1}) \norm{v}_{L^{\beta\lambda^{k}}(D_{k})}~\forall k\geq 1.
\label{secondM}
\end{equation}
In the above inequalities, $C_{0}$ and $C_{1}$ are defined by (\ref{Czero}).
\end{lemma}
\begin{proof}
If $\beta\geq 2$, then let $\Lambda(\beta) =100\beta$. Note that $\eta_{k}\equiv 1$ on $D_{k}$ and $\eta_k \equiv 0$ outside $D_{k-1}$ 
by (\ref{etaDk}). Recall that $S: = M \times [0, T)$. We have
\begin{eqnarray*}
 \norm{v^{\beta}}_{L^{\frac{n +2}{n}}(D_{k})}  &\leq&  \norm{\eta_{k}^2 v^{\beta}}_{L^{\frac{n +2}{n}}(S)}\\
&\leq& C_{a}(n, q, C_{0}, C_{1})\Lambda(\beta)^{1 + \nu}\int_{0}^{T}\int_{M_{t}}v^{\beta}\left(\eta_{k}^2 + \abs{\nabla\eta_{k}}^2 + 
2\eta_{k} \abs{(\frac{\partial}{\partial t}-\Delta)\eta_{k}}\right)d\mu dt\\
&\leq & c_{n}4^{k +1}C_{a}(n, q, C_{0}, C_{1})\Lambda(\beta)^{1 + \nu}\int_{D_{k-1}} v^{\beta} d\mu dt\\
&=&C_{z}(n, q, C_{0}, C_{1}) 4^{k-1} \beta^{1 + \nu}\norm{v^{\beta}}_{L^{1}(D_{k-1})},
\end{eqnarray*}
where $C_{z}(n, q, C_{0}, C_{1}) := 4^2\times 100^{1+\nu}c_{n}C_{a}(n, q, C_{0}, C_{1})$. In the above chain of inequalities, the second line follows from Lemma \ref{softRH}, the third line results from Lemma \ref{heatest} and the 
definition of $\rho_{k}$. For simplicity, let us denote  by $C_{z} = C_{z}(n, q, C_{0}, C_{1}).$ Then the previous estimate can be written in the following form
\begin{equation}
 \norm{v}_{L^{\frac{n +2}{n}\beta}(D_{k})}\leq  C^{\frac{1}{\beta}}_{z}4^{\frac{k-1}{\beta}} 
\beta^{\frac{1 + \nu}{\beta}} \norm{v}_{L^{\beta}(D_{k-1})}.
\label{firstRH}
\end{equation}
This form of the reverse Holder inequality is the key estimate  for our Moser iteration process.  Let $ \lambda =\frac{n +2}{n}$. Then, replacing $\beta$ by $\lambda^{k-1}\beta$ in (\ref{firstRH}), one obtains
\begin{equation}
 \norm{v}_{L^{\beta \lambda^{k}}(D_{k})}\leq  C^{\frac{1}{\beta\lambda^{k-1}}}_{z}4^{\frac{k-1}{\beta\lambda^{k-1}}} 
(\lambda^{k-1}\beta)^{\frac{1 + \nu}{\beta\lambda^{k-1}}} \norm{v}_{L^{\beta\lambda^{k-1}}(D_{k-1})}.
\end{equation}
It follows by iteration that for all $k_{0}\geq 0$
 \begin{eqnarray}
  \norm{v}_{L^{\beta\lambda^{k}}(D_{k})} &\leq&  \left(C^{\frac{1}{\beta}}_{z}\beta^{\frac{1+\nu}{\beta}}\right)^{\sum^{k-1}_{j=k_{0}}
\frac{1}{\lambda^{j}}}
(4^{\frac{1}{\beta}}\cdot \lambda^{\frac{1+\nu}{\beta}})^{\sum^{k-1}_{j=k_{0}}
\frac{j}{\lambda^{j}}} 
\norm{v}_{L^{\beta\lambda^{k_{0}}}(D_{k_{0}})} \nonumber \\
&\leq& C_{b}(n, q, \beta, C_{0}, C_{1})\norm{v}_{L^{\beta\lambda^{k_{0}}}(D_{k_{0}})},
\label{secondRH}
 \end{eqnarray}
where we choose 
\begin{equation}
\label{eq-Cb}
C_{b}(n, q, \beta, C_{0}, C_{1}) = (4\lambda^{1+ \nu} C_z\beta^{1+\nu})^{\frac{n^2}{\beta}},
\end{equation}
since
$$ \sum^{\infty}_{j=0}
\frac{j}{\lambda^{j}} =\frac{\lambda}{(\lambda -1)^2} = \frac{n(n+2)}{4}\leq n^2.$$

Note that  $D_{0} = D$ and $D^{'}\subset D_{k}$ for all positive integer $k$. Thus
\begin{equation}
 \norm{v}_{L^{\beta\lambda^{k}}(D^{'})}\leq C_{b}(n, q, \beta, C_{0}, C_{1}) \norm{v}_{L^{\beta}(D_{0})} = 
C_{b}(n, q, \beta, C_{0}, C_{1}) \norm{v}_{L^{\beta}(D)}.
\end{equation}
Letting $k\rightarrow \infty$, we find that
\begin{equation}
 \norm{v}_{L^{\infty}(D^{'})}\leq C_{b}(n, q, \beta, C_{0}, C_{1}) \norm{v}_{L^{\beta}(D)}.
\end{equation}
If we fix a $k \ge 1$, doing the above iteration process for $l\ge k$, letting $l\to \infty$ we obtain (\ref{secondM}).
\end{proof}

\section{The Moser Iteration Process for the Critical Case}
\h In this section we will deal with the differential inequality
\begin{equation}
 (\frac{\partial}{\partial t}-\Delta ) v\leq fv, ~v\geq 0,
\end{equation}
where the function $f$ is bounded in the $L^{\frac{n +2}{2}}(S)$
norm. In this case $\frac{q}{q-1} =\frac{n +2}{n}$. Thus, we cannot
absorb the term $c_{n}C_{1}\Lambda(\beta)C_{0}\norm{\eta^2
v^{\beta}}_{L^{\frac{q}{q-1}}(S)}$ appearing on the right hand side
of estimate (\ref{superc}) into the left hand side of the same equation, which was
the crucial estimate in obtaining the reverse H\"older inequality
(\ref{firstRH}). This inequality is the key estimate for performing the Moser iteration
in the supercritical case.  However if we assume a smallness condition
on $C_{0} = C_{0}(\frac{n +2}{2})$, we have the following lemma.
\begin{lemma}
\label{lem-Cc}
Let $\beta$ be a constant greater than $1$. Then there exist two constants $\delta_{1}(n, \beta, C_{1})$ and $C_{c}(n,\beta, C_{1})$ such that
 if
\begin{equation}
  \norm{f}_{L^{\frac{n +2}{2}}(D)}\leq \delta_{1} (n,\beta, C_{1})
\label{small1}
\end{equation}
then
\begin{equation}
 \norm{v}_{L^{\frac{n +2}{2}\beta}(D_{1})}\leq C_{c}(n, \beta, C_{1})\norm{v}_{L^{\beta}(D)}.
\label{critest}
\end{equation}
\label{smalllem}
\end{lemma}
\begin{proof}
We will use (\ref{superc}) with $\eta =\eta_{1}$, keeping  in mind that the constant $C_{0}$ appearing on the right hand 
side of (\ref{superc})  can be chosen to be $\norm{f}_{L^{q}(D)}$, where $q =\frac{n +2}{2}$. 
Since $\eta_{1}\equiv 0$ outside $D$, we see that the term $\int_{0}^{s}\int_{M_{t}}\abs{f} \eta^2_{1}v^{\beta} d\mu dt$ can be bounded from
above by 
\begin{equation*}
 \int_{D}\abs{f} \eta^2_{1}v^{\beta} d\mu dt \leq \norm{f}_{L^{q}(D)}\norm{\eta_{1}^2 v^{\beta}}_{L^{\frac{q}{q-1}}(D)}.
\end{equation*}
Thus, for $\eta =\eta_{1}$, the term $C_{0}\norm{\eta_{1}^2 v^{\beta}}_{L^{\frac{q}{q-1}}(M\times [0, T))}$ in $A$ (defined 
in (\ref{superc})) can be replaced by
$\norm{f}_{L^{q}(D)}\norm{\eta_{1}^2 v^{\beta}}_{L^{\frac{q}{q-1}}(D)}.$

Consequently, as $q= \frac{n + 2}{2}$ and $\frac{q}{q-1} =\frac{n +2}{n}$, we have
\begin{multline}
 \norm{\eta_{1}^2 v^{\beta}}_{L^{\frac{n + 2}{n}}(D)}\\ \leq
c_{n}C_{1}\Lambda (\beta )\left(\int_{0}^{T}\int_{M_{t}}  v^{\beta}\left\{
2\eta_{1} \abs{(\frac{\partial}{\partial t} -\Delta) \eta_{1}}  + \abs{\nabla \eta_{1}}^2\right\} d\mu dt + 
 \norm{f}_{L^{\frac{n +2}{2}}(D)}\norm{\eta_{1}^2 v^{\beta}}_{L^{\frac{n+2}{n}}(D)} \right).
\label{criticalc}
\end{multline}

If we choose 
\begin{equation}
\label{eq-delta1}
\delta_{1}(n, \beta, C_{1}) = \frac{1}{2c_{n}C_{1}\Lambda (\beta)},
\end{equation}
then  from (\ref{small1}) and (\ref{criticalc}) we get
\begin{eqnarray*}
 \norm{\eta_{1}^2 v^{\beta}}_{L^{\frac{n + 2}{n}}(D)}\leq
c_{n}C_{1}\Lambda (\beta )\left(\int_{0}^{T}\int_{M_{t}}  v^{\beta}\left\{
2\eta_{1} \abs{(\frac{\partial}{\partial t} -\Delta) \eta_{1} } + \abs{\nabla \eta_{1}}^2\right\}d\mu dt \right).
\end{eqnarray*}
By Lemma \ref{heatest},
\begin{equation}
 \sup \left(\eta^2_{1} + \abs{\nabla\eta_{1}}^2 + 
2\eta_{1} \abs{(\frac{\partial}{\partial t}-\Delta)\eta_{1}}\right)\leq \frac{c_{n}}{\rho^2_{1}} = c_{n}.
\end{equation}
Hence
\begin{equation}
  \norm{\eta_{1}^2 v^{\beta}}_{L^{\frac{n + 2}{n}}(D)} \leq c_{n}C_{1}\Lambda(\beta)
\norm{v^{\beta}}_{L^{1}(D)}.
\end{equation}
Now, recalling that $\eta_{1}\equiv 1$ on $D_{1}$, we have
\begin{eqnarray*}
  \norm{ v}^{\beta}_{L^{\frac{n + 2}{n}\beta}(D_{1})} =  \norm{v^{\beta}}_{L^{\frac{n + 2}{n}}(D_{1})}\leq  
\norm{\eta_{1}^2 v^{\beta}}_{L^{\frac{n + 2}{n}}(D)}\leq c_{n}C_{1}\Lambda(\beta)\norm{v^{\beta}}_{L^{1}(D)}
= c_{n}C_{1}\Lambda(\beta)\norm{v}^{\beta}_{L^{\beta}(D)}.
\end{eqnarray*}
Thus
\begin{equation*}
 \norm{v}_{L^{\frac{n +2}{2}\beta}(D_{1})}\leq C_{c}(n, \beta, C_{1})\norm{u}_{L^{\beta}(D)},
\end{equation*}
where 
\begin{equation}
\label{eq-Cc}
C_{c}(n,\beta, C_{1}) := \left(c_{n}C_{1}\Lambda(\beta)\right)^{\frac{1}{\beta}}.
\end{equation}

\end{proof}
\section{Bounding the Mean Curvature}

\h In this section we will bound the mean curvature along the mean
curvature flow in terms of $||H||_{L^{n+2}(S)}$, having that the
second fundamental form is uniformly bounded from below and an extra
smallness assumption. First we derive a differential inequality for
the modified mean curvature $\hat{H}$ defined below.
\begin{lemma}
 Suppose that
\begin{equation}
 h_{ij}\geq -B g_{ij}.
\end{equation}
Then, for 
\begin{equation*}
 \hat{H} := H + nB\geq 0
\end{equation*}
one has
\begin{equation}
 \partial_{t}\hat{H}\leq \Delta \hat{H} + \hat{H}^3 + nB^2 \hat{H}.
\end{equation}
\label{lowerH}
\end{lemma}
\begin{proof}
Recall that the evolution equation of the mean curvature $H$ (\cite{Huisken84}) is,
\begin{equation}
 \partial_{t}H = \Delta H + \abs{A}^2 H.
\label{evolH}
\end{equation}
Let $\lambda_{i} (i=1,\cdots, n)$ be the principle curvatures. Then $\lambda_{i}\geq -B$ and
\begin{equation*}
 \abs{A}^2 = \lambda_{1}^2 + \cdots + \lambda_{n}^2.
\end{equation*}
Because $\lambda_{i} + B\geq 0$, one gets
\begin{equation*}
 (\lambda_{1} + B )^2 + \cdots + (\lambda_{n} + B)^2\leq (\lambda_{1} + B + \cdots + \lambda_{n} + B)^2 = (H + nB)^2 =\hat{H}^2.
\end{equation*}
That is
\begin{equation*}
 \abs{A}^2 + 2BH + nB^2 \leq \hat{H}^2
\end{equation*}
or, equivalently
\begin{equation}
 \abs{A}^2 \leq \hat{H}^2 - 2BH -nB^2 = \hat{H}^2 -2B(\hat{H}-nB) -nB^2 = \hat{H}^2 -2B\hat{H} + nB^2.
\label{upperA}
\end{equation}
Now, by (\ref{evolH}), we have
\begin{eqnarray*}
 \partial_{t}\hat{H} &=& \partial_{t}H = \Delta H + \abs{A}^2 H =\Delta\hat{H} + (\hat{H}-nB)\abs{A}^2\\
&\leq & \Delta\hat{H} + \hat{H}(\hat{H}^2 -2B\hat{H} + nB^2)\\
&\leq & \Delta\hat{H} + \hat{H}(\hat{H}^2 + nB^2).
\end{eqnarray*}
\end{proof}

Let $C_0, C_1$ be as in (\ref{Czero}) and $C_c$ as in Lemma \ref{lem-Cc}.
Then, using Moser iteration, we can establish a Harnack type inequality for the mean curvature.
\begin{lemma}
 Suppose that
\begin{equation}
 h_{ij}\geq -B g_{ij}.
\end{equation}
 Let
\begin{equation}
C_{3} := 2C^{2}_{c}(n, n +2, C_{1})\left(\norm{H}^2_{L^{n +2}(D)} + n^2B^2 \norm{1}^2_{L^{n +2}(D)}+ nB^2 \norm{1}_{L^{\frac{(n +2
)^2}{2n}}(D)}\right),
\label{Cthree}
\end{equation}
then there exist positive
constants $\delta_{2}(n, C_{1})$ and $C_{d}(n, C_{1}, C_{3})$ such
that if
\begin{equation}
 \norm{H}_{L^{n +2}(D)} + B\norm{1}_{L^{n +2}(D)} \leq \delta_{2} (n, C_{1})
\end{equation}
then
\begin{equation}
 \norm{H^{+}}_{L^{\infty}(D^{'})}\leq C_{d}(n, C_{1}, C_{3})(\norm{H}_{L^{n +2}(D)} + B\norm{1}_{L^{n +2}(D)}).
\end{equation}
\label{MoserMC}
\end{lemma}
\begin{proof}
Let $\hat{H} = H + nB\geq 0$. Then, by Lemma \ref{lowerH}
\begin{equation} 
 \partial_{t}\hat{H}\leq \Delta\hat{H} + f\hat{H}, \qquad \mbox{where} \qquad
 f = \hat{H}^2 + nB^2.
\label{evolhatH}
\end{equation}
We have
\begin{eqnarray*}
 \norm{f}^{\frac{n +2}{2}}_{L^{\frac{n +2}{2}}(D)} &=&\int_{0}^{1}\int_{M_{t}\cap B(x_0,1)} \left(\hat{H}^2 + nB^2 \right)^{\frac{n +2}{2}}
d\mu\, dt\\
&\leq& c_{n}\int_{0}^{1}\int_{M_{t}\cap B(x_0,1)} (\hat{H}^{n +2} + B^{n +2})d\mu dt \\
&\leq& c_{n}\int_{0}^{1}\int_{M_{t}\cap B(x_0,1)} (|H|^{n +2} + B^{n +2})d\mu dt\\
&=& c_{n}\norm{H}^{n + 2}_{L^{n +2}(D)} + c_{n}B^{n +2} \norm{1}_{L^{n +2}(D)}^{n +2}\\
&\leq& c_{n} \left(\norm{H}_{L^{n +2}(D)} + B\norm{1}_{L^{n +2}(D)} \right)^{n +2}.
\end{eqnarray*}
It follows that
\begin{equation*}
  \norm{f}_{L^{\frac{n +2}{2}}(D)} \leq c_{n}\left(\norm{H}_{L^{n +2}(D)} + B\norm{1}_{L^{n +2}(D)} \right)^{2}.
\end{equation*}
Let us choose
\begin{equation}
\label{eq-delta2}
 \delta_{2}(n, C_{1}) = \frac{\delta^{\frac{1}{2}}_{1}(n, n +2, C_{1})}{c_{n}^{\frac{1}{2}}},
\end{equation}
where $\delta_1(n,\beta,C_1)$ has been defined in the proof of Lemma \ref{smalllem}. Then, if
\begin{equation*}
 \norm{H}_{L^{n +2}(D)} + B\norm{1}_{L^{n +2}(D)} \leq \delta_{2} (n, C_{1}),
\end{equation*}
one has
\begin{equation*}
 \norm{f}_{L^{\frac{n +2}{2}}(D)} \leq \delta_{1}(n, n +2, C_{1}).
\end{equation*}
This means we are in the critical case ($q = \frac{n+2}{2}$) having a smallness assumption (\ref{small1}) satisfied. Hence, 
we can apply (\ref{critest}) with $\beta = n +2$ to obtain
\begin{equation}
\label{eq-H-sc}
 \norm{\hat{H}}_{L^{(\frac{n +2}{n}) (n +2)}(D_{1})} \leq C_{c}(n, n +2, C_{1})\norm{\hat{H}}_{L^{n +2}(D)}.
\end{equation}
This inequality brings us to the supercritical case for  (\ref{evolhatH}). In fact, let 
\begin{equation*}
 q = (\frac{n +2}{n})\cdot(\frac{n +2}{2}) > \frac{n +2}{2}.
\end{equation*}
Then we can bound $f$ in $\norm{\cdot}_{L^{q}(D_{1})}$ by $C_{3}$ defined by (\ref{Cthree}). Indeed, using 
(\ref{eq-H-sc}) we get
\begin{eqnarray*}
 \norm{f}_{L^{q}(D_{1})} &=& \norm{\hat{H}^2 + nB^2}_{L^{q}(D_{1})}\\
&\leq & \norm{\hat{H}}^{2}_{L^{(\frac{n +2}{n})(n +2)}(D_{1})} + nB^2 \norm{1}_{L^{\frac{(n +2
)^2}{2n}}(D_{1})}\\
&\leq & C^{2}_{c}(n, n +2, C_{1})\norm{\hat{H}}^2_{L^{n +2}(D)} + nB^2 \norm{1}_{L^{\frac{(n +2
)^2}{2n}}(D)}\\ &\leq&
2C^{2}_{c}(n, n +2, C_{1})\left(\norm{H}^2_{L^{n +2}(D)} + n^2B^2 \norm{1}^2_{L^{n +2}(D)}+ nB^2 \norm{1}_{L^{\frac{(n +2
)^2}{2n}}(D)}\right)
.
\end{eqnarray*}
Thus, we can use (\ref{secondM}) with $\lambda =\frac{n +2}{n}$, 
$\beta = n + 2$ and $k=1$ to obtain
\begin{eqnarray*}
 \norm{\hat{H}}_{L^{\infty}(D^{'})}&\leq& C_{b}(n, q, n +2, C_{3}, C_{1}) \norm{\hat{H}}_{L^{\frac{n +2}{n}(n +2)}(D_{1})}\\ &\leq& 
 C_{b}(n, q, n +2, C_{3}, C_{1})C_{c}(n, n+2, C_{1})\norm{\hat{H}}_{L^{n +2}(D)}.
\end{eqnarray*}
Noting that
\begin{equation*}
 \norm{H^{+}}_{L^{\infty}(D^{'})}\leq \norm{\hat{H}}_{L^{\infty}(D^{'})}
\end{equation*}
we finally obtain the desired estimate
\begin{eqnarray*}
 \norm{H^{+}}_{L^{\infty}(D^{'})}&\leq& C_{b}(n, q, n +2, C_{3}, C_{1})C_{c}(n, n +2, C_{1})\norm{\hat{H}}_{L^{n +2}(D)}\\
&\leq &n C_{b}(n, q, n +2, C_{3}, C_{1})C_{c}(n, n +2, C_{1}) (\norm{H}_{L^{n +2}(D)} + B\norm{1}_{L^{n +2}(D)})\\
&= & C_{d}(n, C_{1}, C_{3})(\norm{H}_{L^{n +2}(D)} + B\norm{1}_{L^{n +2}(D)}),
\end{eqnarray*}
where 
\begin{equation}
\label{eq-Cd}
C_d(n, C_{1}, C_{3}) := n C_{b}(n, q, n +2, C_{3}, C_{1})C_{c}(n, n +2, C_{1}).
\end{equation}
\end{proof}
\section{Proof of the Main Theorem}
\label{thmproof}
\h In this section, we give the proof of the main Theorem \ref{MCbound} stated in the introduction. 

\begin{proof}[Proof of Theorem \ref{MCbound}]
Since  $\norm{H}_{L^{\alpha}(M\times [0, T))}< \infty$ implies that $\norm{H}_{L^{ n + 2}(M\times [0, T))}<\infty$ if $\alpha> n +2$,
we only need to prove the Theorem for $\alpha = n +2$. \\
\h We argue by contradiction. Suppose that $T$ is the extinction time of the flow. Then, 
by Theorem \ref{Aunbound}, $\abs{A}$ is unbounded. It follows from (\ref{pinchA}) that we have (\ref{upperA}),
\begin{equation*}
 \abs{A}^2 \leq \hat{H}^2 -2B\hat{H} + nB^2 \leq \hat{H}^2 + nB^2 = (H + nB)^2 + nB^2,
\end{equation*}
and thus $\abs{H}$ is unbounded.
Because $H\geq -nB$, we know that $H^+$ is unbounded. Therefore, there exists a sequence of points $(x_{i}, t_{i})$ with $x_{i}\in M_{t_{i}}$ 
such that
\begin{equation}
 Q_{i}: = H(x_{i}, t_{i}) = \max_{0\leq t\leq t_{i}} \max_{x\in M_{t}} H(x, t) \to +\infty.
\label{maxMC}
\end{equation}
Consider the sequence $\tilde{M}^{i}_{t}$ of rescaled solutions for $t\in [0,1]$ defined by
\begin{equation*}
 \tilde{F}_{i} (\cdot,t) = Q_{i} F(\cdot, t_{i} + \frac{t-1}{Q^2_{i}}).
\end{equation*}
If $g, H$ and $A:= \{h_{jk}\}$ are the induced metric, the mean curvature and the second fundamental form of $M_{t}$, respectively, then
the corresponding rescaled quantities are given by
\begin{equation*}
 \tilde{g}_{i} = Q^2_{i} g;~ \tilde{H}_{i} =\frac{H}{Q_{i}};~ \abs{\tilde{A}_{i}}^2 =\frac{\abs{A}^2}{Q^2_{i}}.
\end{equation*}
It follows from (\ref{maxMC}) and (\ref{pinchA}) that, for the rescaled solutions we have
\begin{equation*}
 \tilde{H}_{i}(x_{i}, 1) =1
\end{equation*}
and
\begin{equation*}
 \tilde{A}_{i}\geq -\frac{B}{Q_{i}}\tilde{g}_{i}.
\end{equation*}
Consider the following sets in space and time
\begin{equation*}
 \tilde{D}^{i} = \cup_{0\le t\le 1} (B(x_{i}, 1)\cap (\tilde{M}_{i})_t);\,\,\, 
(\tilde{D}^{i})^{'} = \cup_{\frac{1}{12}\le t< 1}(B(x_{i}, \frac{1}{2})\cap (\tilde{M}_{i})_t).
\end{equation*}
Then, we can calculate
\begin{multline}
 \lim_{i\rightarrow\infty }\left( \norm{\tilde{H}_{i}}_{L^{n +2}(\tilde{D}^{i})} +\frac{B}{Q_{i}}\norm{1}_{L^{n +2}(\tilde{D}^{i})}\right)\\ =
\lim_{i\rightarrow\infty }\left\{ \left( \int^{t_{i}}_{t_{i}-\frac{1}{Q^2_{i}}}\int_{M_{t}\cap B(x_{i},\frac{1}{Q_{i}})}\abs{H}^{n +2} d\mu dt
\right)^{\frac{1}{n +2}} + B \left( \int^{t_{i}}_{t_{i}-\frac{1}{Q^2_{i}}}\int_{M_{t}\cap B(x_{i},\frac{1}{Q_{i}})} d\mu dt
\right)^{\frac{1}{n +2}}\right\}\\
\leq \lim_{i\rightarrow\infty }\left\{ \left( \int^{t_{i}}_{t_{i}-\frac{1}{Q^2_{i}}}\int_{M_{t}}\abs{H}^{n +2} d\mu dt
\right)^{\frac{1}{n +2}} + B \left( \int^{t_{i}}_{t_{i}-\frac{1}{Q^2_{i}}}\int_{M_{t}} d\mu dt
\right)^{\frac{1}{n +2}}\right\} 
= 0.
\label{vanishH}
\end{multline}
The last step follows from the facts that
\begin{equation*}
  \int_{0}^{T}\int_{M_{t}}\abs{H}^{n +2} d\mu dt<\infty; \,\,\,\, \int_{0}^{T}\int_{M_{t}}d\mu dt<\infty
\end{equation*}
and
\begin{equation*}
 \lim_{i\rightarrow \infty} \frac{1}{Q^2_{i}} =0.
\end{equation*}
Consequently, there is a universal constant $C>1$ such that, for our rescaled flows, the constants 
\begin{equation*}
\tilde{C}^{i}_{1} = (1 +\norm{\tilde{H}_{i}}^{n + 2}_{L^{n + 2}(\tilde{M}_{i}\times [0,1])})^{\frac{n}{n + 2}} = 
\left(1 +  \int^{t_{i}}_{t_{i}-\frac{1}{Q^2_{i}}}\int_{M_{t}}\abs{H}^{n +2} d\mu dt \right)^{\frac{n}{n + 2}}
\end{equation*}
satisfy
\begin{equation}
\label{eq-C1bound}
 1\leq \tilde{C}^{i}_{1}\leq C.
\end{equation}
By our choice of the constants $\delta_{2}(n, \tilde{C}^i_1)$, which are decreasing in the second variable (follows from (\ref{eq-delta1})
and (\ref{eq-delta2})), we obtain
\begin{equation*}
 \delta_{2}(n, \tilde{C}^{i}_{1})\geq \delta_{2}(n, C)>0. 
\end{equation*}
\h Hence, recalling (\ref{vanishH}), we have for $i$ sufficiently large,
\begin{equation}
  \norm{\tilde{H}_{i}}_{L^{n +2}(\tilde{D}^{i})} +\frac{B}{Q_{i}}\norm{1}_{L^{n +2}(\tilde{D}^{i})} \leq \delta_{2}(n, \tilde{C}^{i}_{1}).
\end{equation}
Thus, by Lemma \ref{MoserMC}
\begin{equation}
 \norm{\tilde{H}_{i}^{+}}_{L^{\infty}((\tilde{D}^{i})^{'})} \leq C_{d}(n, \tilde{C}^{i}_{1}, \tilde{C}^{i}_{3})
(\norm{\tilde{H}_{i}}_{L^{n +2}(\tilde{D}^{i})} +\frac{B}{Q_{i}}\norm{1}_{L^{n +2}(\tilde{D}^{i})}).
\label{Moserrescaled}
\end{equation}
On the one hand, we can also check that, for our rescaled flows, the constants 
\begin{equation*}
\tilde{C}^{i}_{3}= 2C^{2}_{c}(n, n +2, \tilde{C}^{i}_{1})\left(\norm{\tilde{H}_{i}}^2_{L^{n +2}(\tilde{D}^{i})} + 
n^2\frac{B^2}{Q^2_{i}} \norm{1}^2_{L^{n +2}(\tilde{D}^{i})}+ n\frac{B^2}{Q^2_{i}} \norm{1}_{L^{\frac{(n +2
)^2}{2n}}(\tilde{D}^{i})}\right)
\end{equation*}
satisfy
\begin{equation}
\label{eq-C3bound}
 \tilde{C}^{i}_{3}\leq C.
\end{equation}
This easily follows from (\ref{eq-Cc}) and (\ref{eq-C1bound}).
On the other hand, by our choice of the constants $C_{d}(n, \cdot, \cdot)$, which are increasing in the second and third variables 
(by (\ref{eq-Cd}), (\ref{eq-Cb}), (\ref{eq-Cc}) and (\ref{Cthree})), since we have (\ref{eq-C1bound}) and
(\ref{eq-C3bound}) we obtain
\begin{equation*}
 C_{d}(n, \tilde{C}^{i}_{1}, \tilde{C}^{i}_{3})\leq C_{d}(n, C, C) <\infty. 
\end{equation*}
As a result, we deduce from (\ref{Moserrescaled}) that
\begin{equation}
 \norm{\tilde{H}_{i}^{+}}_{L^{\infty}((\tilde{D}^{i})^{'})} \leq C_{d}(n, C, C)
\left(\norm{\tilde{H}_{i}}_{L^{n +2}(\tilde{D}^{i})} +\frac{B}{Q_{i}}\norm{1}_{L^{n +2}(\tilde{D}^{i})}\right).
\label{Moserrescaled2}
\end{equation}
Letting $i\rightarrow \infty$ in (\ref{Moserrescaled2}), and recalling (\ref{vanishH}), 
we find that
\begin{equation}
\lim_{i\rightarrow \infty} \norm{\tilde{H}_{i}^{+}}_{L^{\infty}((\tilde{D}^{i})^{'})} =0.
\end{equation}
This is a contradiction because $\norm{\tilde{H}_{i}^{+}}_{L^{\infty}((\tilde{D}^{i})^{'})} \geq \tilde{H}_{i}(x_{i},1) =1$ for all $i$. The proof
of our Main Theorem is complete.
\end{proof}

We will give a proof of Corollary \ref{cor-appl}.

\begin{proof}[Proof of Corollary \ref{cor-appl}]
In both cases, the mean convex case and the starshaped case we have\\
 (i) A lower bound for the mean curvature
\begin{equation*}
 H\geq -l~ \text{for some} ~l\ge0
\end{equation*}
and \\
(ii) An upper bound for the squared second fundamental form in terms of a linear function of the squared mean curvature
\begin{equation*}
 \abs{A}^2 \leq C_{\ast}H^2 + b~\text{for some}~ C_{\ast}, b>0.
\end{equation*}
In the mean convex case (ii) follows from \cite{HS} and in the
starshaped case, both (i) and (ii) follow from \cite{Smoczyk}, for
some uniform constants $C_{\ast}, b, l$. Choose $k$ large enough so that
$k > 2l$ and $(k-l)^2 \ge b$ which imply, for $\tilde{H} = H + k$,  
$$\tilde{H} > l \,\,\, \Longrightarrow (\tilde{H}-k)^2 \le \tilde{H}^2,$$
$$\tilde{H}^2 \ge (k-l)^2 \ge b, \,\,\, \mbox{and therefore},$$
$$\abs{A}^2 \leq C_{\ast}H^2 + b = C_{\ast}(\tilde{H}-k)^2 + b \le  \tilde{C}_{\ast}\tilde{H}^2,$$
for a uniform constant $\tilde{C}_{\ast}$. It easily follows that
\begin{equation*}
 \partial_{t}\tilde{H}\leq \Delta \tilde{H} + \tilde{C}_{\ast}\tilde{H}^3.
\end{equation*}
As can be seen from the proofs of Lemma \ref{MoserMC} and Theorem 8.1, this differential inequality combined with the integral bound 
(\ref{intboundH}) of the mean curvature in 
Theorem 1.2 allows us to extend the mean curvature flow past time $T$.
\end{proof}

{} 

\end{document}